\documentclass[12pt]{article}

\usepackage[utf8]{inputenc}
\usepackage{amsmath, amssymb, amsfonts, amsthm, amscd}
\usepackage{manfnt}
\usepackage{mathtools}
\usepackage{fullpage}
\usepackage{url}
\usepackage{txfonts}
\usepackage{caption}
\usepackage[shortlabels]{enumitem}
\usepackage{hyperref}
\usepackage{arraycols}
\usepackage{mathtools,amssymb}
\usepackage{mathrsfs}
\usepackage{dirtytalk}
\usepackage{biblatex}
\usepackage{tikz}
\newtheorem{theorem}{Theorem}[section]
\newtheorem{corollary}[theorem]{Corollary}
\newtheorem{proposition}[theorem]{Proposition}
\newtheorem{lemma}[theorem]{Lemma}
\theoremstyle{definition}
\newtheorem{definition}[theorem]{Definition}
\newtheorem{example}[theorem]{Example}
\theoremstyle{remark}
\newtheorem*{remark}{Remark}

\newcommand{\Z}{\mathbb{Z}}
\newcommand{\D}{\mathbb{D}}

\newcommand{\R}{\mathbb{R}}

\newcommand{\N}{\mathbb{N}}

\renewcommand{\H}{\mathbb{H}}

\renewcommand{\emptyset}{\varnothing}

\title{Curtain Model for CAT(0) Spaces and Isometries}

\author{Yutong Chen \\ \href{mailto:yutong03@mit.edu}{yutong03@mit.edu}}
\date{}

\begin{document}

\maketitle

\begin{abstract}
    This paper studies the dynamics of isometries in the curtain model, which is used to capture the hyperbolicity in a fixed CAT(0) space. We establish several fundamental properties and fully classify the behavior of semisimple isometries of a CAT(0) space in the associated curtain model. In the nonsemisimple case, we restrict the behavior of parabolic actions with positive translation length in the curtain model in most cases of interest, allowing the use of ping-pong-like techniques on the curtain model to provide insights into the study of CAT(0) groups. 
\end{abstract}
\tableofcontents

\section{Introduction}

In this paper, we will work in the setting of CAT(0) spaces and study their isometry groups. CAT(0) spaces are a generalization of simply connected Riemannian manifolds with nonpositive sectional curvatures, where the \say{nonpositive curvature} condition is given by the CAT(0) inequality on saying that triangles are no fatter than in the corresponding triangle in the Euclidean space. In this setting, the natural class of groups we hope to understand are CAT(0) groups. We say a group $G$ is CAT(0) if $G$ acts geometrically on a proper CAT(0) space.

Despite the success of the program of geometric group theory, there remain fundamental algebraic questions concerning groups with interesting geometric properties that have yet to be fully answered. One such question is the behavior of subgroups, which can be more specifically phrased in understanding the algebraic properties of group elements starting from their geometric properties. A well-known property of this kind is the Tits Alternative. We say a group $G$ satisfies the Tits alternative if for every finitely generated subgroup $H$ of $G$,  either $H$ is virtually solvable or it contains a non-abelian free subgroup. It's been shown that a wide range of groups satisfies the Tits alternative, including finitely generated linear groups over a given field \cite{Tits1972FreeSI}, hyperbolic groups \cite{gromov1987hyperbolic}, mapping class groups \cite{mccarthy1985tits}, and for groups acting geometrically on CAT(0) cubical complexes \cite{sageev2004titsalternativecat0cubical}.

 However, Tits alternative is still completely open for general CAT(0) groups. Even an a priori simpler question on whether there could be infinite torsion subgroups in CAT(0) groups has solutions only in special cases and still uses advanced techniques \cite{norin2022torsion}.  We will study the action of the isometry group on the associated curtain model, which is introduced by Petyt-Spriano-Zalloum as an analog of the curve graph of a surface \cite{petyt2023hyperbolic}. Since the curtain model is hyperbolic, one can hope to use it to make progress towards the Tits alternative, since, by the ping-pong lemma, subgroups that act non-elementarily on a hyperbolic space need to have a free subgroup. So, it is natural to wonder about actions without two independent loxodromic elements.

The goal of this paper is to fully describe how the dynamics of an isometry changes after passing to the associated curtain model. To start with, we first recall the classifications of isometries of CAT(0) spaces and of hyperbolic spaces (see section 2.3 for more details). 

\begin{definition}[Translation Length]
Let X be a metric space, and $g\in$ Isom($X$). We define the \emph{translation length} of $g$ as:
\[
\tau(g)=\lim_{n\rightarrow \infty}\frac{d(x,g^nx)}{n}.
\]
Then $\tau(g)$ is well-defined and is independent of the choice of base point $x$. Note that $\tau(g)\leq \inf_{x\in X} d(x,gx)$.
Thus we also define the \emph{minimal set}:
\[
\textup{Min}(g)=\{x\in X: d(x,gx)=\tau(g)\}.
\]
Call the isometry $g$ \emph{semisimple} if $\textup{Min}(g)\neq \emptyset$.
\end{definition}

\begin{definition}
[Classification of Isometries of CAT(0) Spaces]
Let $X$ be a CAT(0) space, and $g\neq id$ be an isometry of $X$. Then precisely one of the following hold:
\begin{enumerate}
    \item Call $g$ \emph{elliptic} if it has a fixed point.
    \item Call $g$ \emph{parabolic} if $\text{Min}(g)=\emptyset$.
    \item Call $g$ \emph{hyperbolic} if $\text{Min}(g) \neq \emptyset$ and $\tau(g)>0$; equivalently, there is a geodesic line $l$ such that $g$ acts on $l$ by translation of $\tau(g)$.
\end{enumerate}
\end{definition}

\begin{definition}
[Classification of Isometries of Hyperbolic Spaces]
    Let $X$ be a hyperbolic space, and $g\neq id$ be an isometry of $X$. Then precisely one of the following hold:
    \begin{enumerate}
        \item Call $g$ \emph{elliptic} if one (equivalently every) orbit of $\langle g\rangle$ is bounded.
        \item Call $g$ \emph{parabolic} if $\tau(g)=0$ and one (equivalent every) orbit of $\langle g\rangle$ is unbounded.
        \item Call $g$ \emph{loxodromic} if $\tau(g)>0$.
    \end{enumerate}
\end{definition}

We say a parabolic isometry of a CAT(0) space is \emph{strict} if it has $0$ translation length. We also give the notion of a \emph{contracting isometry} of a CAT(0) space, which precisely captures when an isometry becomes loxodromic in the curtain model. 

\begin{definition}
Let $X$ be a CAT(0) space, $A\subset X$, and $g\in$ Isom($X$). The subset $A$ is called \emph{$C-\text{contracting}$}, if for any open ball $B$ disjoint from $A$, a closest point projection $\pi_A(B)$ has diameter at most $C$. We call a set contracting if it's $C$-contracting for some $C$. We call $g$ a \emph{contracting isometry} if there exists some $x\in X$ such that:
\begin{enumerate}
    \item $n\mapsto g^nx$ is a quasi-isometric embedding from $\Z$ to $X$.
    \item the orbit $\{g^nx:n\in \Z\}$ is $C$-contracting for some $C$.
\end{enumerate}
\end{definition}

Our main theorem is the following on excluding parabolic isometries in the curtain model:

\begin{theorem}
\label{premain}
Let $X$ be a complete and proper CAT(0) space. Fix a curtain model $X_D$ of $X$. Let $g\in$ Isom($X$). If $g$ is parabolic in $X_D$, then $g$ is strictly parabolic in $X$.
\end{theorem}

In fact, we will prove a more general version of Theorem \ref{premain}:

\begin{theorem}
\label{main}
Let $X$ be a CAT(0) space. Fix a curtain model $X_D$ of $X$. Let $g\in$ Isom($X$). Then the following hold:
\begin{enumerate}
    \item $g$ acts loxodromically on $X_D$ if and only if $g$ is contracting in $X$.
    \item If $g$ is semisimple and not contracting in $X$, then $g$ is elliptic in $X_D$.
    \item Suppose further $X$ is complete and proper. If $g$ is a parabolic isometry with $\tau(g)>0$, then $g$ is elliptic in $X_D$. 
\end{enumerate}
\end{theorem}

This is essentially the best result of this type one can hope for. Even though a strictly parabolic isometry cannot be loxodromic in the curtain model, we cannot easily distinguish whether it becomes parabolic or elliptic. For example, for $g$ is parabolic in $X=\H^2$, $g$ remains parabolic in $X_D$ by Theorem \ref{hyperbolicity}; on the other hand, Examples \ref{hyperbolic_to_elliptic} and \ref{parabolic_to_elliptic} show that any kind of isometry in $X$ can become elliptic in $X_D$.

Finally, by examining the proof of Theorem \ref{main}, we can easily deduce the following corollary: 

\begin{corollary}
     Let $X$ be a CAT(0) space, and $g$ is a parabolic isometry with $\tau(g)>0$. Then:
    \begin{itemize}
        \item $g$ is contracting if and only if there exists $x\in X$ such that $[x,g^nx]$ are uniformly contracting as $n\rightarrow \infty$ .
        \item If we further assume $X$ is complete and proper, then $g$ is not contracting.
    \end{itemize}
\end{corollary}

The by-product gives us insights into checking when a parabolic isometry with positive translation length is contracting, which is a property not related to the curtain model.

\subsection*{Acknowledgement}
I owe many thanks to my summer project supervisor Davide Spriano for his invaluable guidance and mentorship. Davide not only provided the initial questions that inspired this paper but also helped me obtain the necessary background to solve these questions. Our weekly discussions helped me resolve many misunderstandings and offered me new ideas to work with. I would also like to thank Harry Petyt, whose generous assistance and supply of related examples greatly facilitated my learning process. Furthermore, I’m grateful to the University of Oxford for providing partial financial support for this project.

\section{Background}

\subsection{Construction of the Curtain Model}
We first briefly recall the associated curtain model to any CAT(0) space, which is due to Petyt, Spriano and Zalloum in \cite{petyt2023hyperbolic}.

\begin{definition}[Curtains]
Let $X$ be a CAT(0) space and $\alpha: I \rightarrow X$ be a geodesic. For a number $r$ with $[r-1/2,r+1/2]$ in the interior of $I$, the \emph{curtain} dual to $\alpha$ at $r$ is

\[
h=h_{\alpha,r}=\pi_\alpha^{-1}([r-1/2,r+1/2]),
\]
where $\pi_\alpha$ is the closest point projection to the geodesic $\alpha$, and the line segment $[r-1/2,r+1/2]$ is called the \emph{pole} of the curtain.

\end{definition}

\begin{definition} [Half Spaces, Separation] Let $X$ be a CAT(0) space and $h$ be a curtain dual to a geodesic $\alpha:I\rightarrow X$. Then the \emph{half spaces} determined by $h$ are $h^{-}=\pi_\alpha^{-1}(-\infty,r-1/2)$ and $h^{+}=\pi_\alpha^{-1}(r+1/2,\infty)$. We say $h$ \emph{separates} subsets $A$ and $B$ if $A \subseteq h^-$ and $B \subseteq h^+$ or vice versa.

\end{definition}

\begin{remark}
    Note that $\{h^{-},h,h^{+}\}$ forms a partition of space $X$. It's also not hard to see that $h^{-},h,h^{+}$ are all connected, and $d(h^{-},h^{+})=1$.
\end{remark}

\begin{definition} [Chains]
A set $\{h_i\}$ of curtains is called a \emph{chain} if $h_i$ separates $h_{i-1}$ from $h_{i+1}$ for all $i$. We say $\{h_i\}$ separates 
$A,B \subseteq X$ if each $h_i$ does. 
For $x\neq y$, define the chain distance as $d_{\infty}(x,y)=1+\max\{|c|:\text{$c$ is a chain separating $x$ from $y$}\}$; and set $d(x,x)=0$. One can verify that $d_\infty$ is indeed a metric and $d_{\infty}(x,y)=\lceil d(x,y) \rceil$.
\end{definition}

\begin{definition}[$L$-separation, $L$-chain]
Disjoint curtains $h$ and $h'$ are called \emph{$L$-separated} if every chain meeting both $h$ and $h'$ has cardinality at most $L$. Call a chain $c$ of curtains $L$-separated if each pair of curtains is $L$-separated. We refer to $c$ as an \emph{$L$-chain}.
\end{definition}

\begin{definition}[$L$-metric]
Fix $L\in \N$. Given distinct points $x,y \in X$, set $d_L(x,x)=0$ and define:
\[
d_{L}(x,y)=1+\max\{|c|:\text{$c$ is a $L$-chain separating $x$ from $y$}\}.
\]
One can verify that $d_L$ is indeed a metric. We denote the metric space $(X,d_L)$ as $X_L$.
\end{definition}

The most important property of the metric space $X_L$ is that it is hyperbolic in the sense of Gromov's four-point condition and it's weakly roughly geodesic. It's also easy to see that Isom($X$)$\leq$Isom($X_L$) for any $L$. In fact, we will later show that equality holds in many reasonable cases.

Finally, we recall the curtain model: a new metric $D$ that combines the information of all of $d_L$. We fix a sequence $\lambda_L \in (0,1)$ such that 

\[
\sum_{L=1}^{\infty}\lambda_L=1, \quad
\sum_{L=1}^{\infty}L^2\lambda_L=\Lambda<\infty.
\]

For $x,y \in X$, define $D(x,y)=\sum_{L=1}^{\infty}\lambda_Ld_L(x,y)$. The function $D$ is clearly a metric, whose triangle inequality inherits directly from the ones of each $d_L$. Call the metric space $X_D=(X,D)$ \emph{the curtain model} of the CAT(0) space $X$. From the definition, one can reasonably conjecture that $X_D$ inherits lots of nice properties from $X_L$, and we'll see later that it's indeed the case.

\subsection{Properties of the Curtain Model}

In this section, we first state a few important properties of the curtain model.

\begin{lemma}[Gluing $L$-chains, Lemma 2.13 of \cite{petyt2023hyperbolic}]
Suppose that $c=\{\cdots,h_0\}$ and $c'=\{h'_1,\cdots\}$ are $L$-chains with $|c|>1$ and $|c'|>L+1$,oriented in the natural way. 
If $h_0^{+} \cap h'_j \neq \emptyset$ for all $j$ and ${h_1'}^-
\cap h_i \neq \emptyset$ for all $i$, then $c''=\{\cdots,h_{-1},h'_{L+2},\cdots\}$ is an $L$-chain. Its cardinality is $|c|+|c'|-L-2$ when $c$ and $c'$ are finite.
\end{lemma}

\begin{proposition}[Dualizing Chains, Lemma 2.21 of \cite{petyt2023hyperbolic}]
\label{dualize}
    Let $L,n\in \N$, and $\{h_1,\cdots,h_{(4L+10)n}\}$ be an $L$-chain. Suppose that $A,B \subset X$ are separated by every $h_i$. For any $x\in A$ and $y\in B$, the sets $A$ and $B$ are separated by an $L$-chain of length at least $n+1$ all of whose elements are dual to $[x,y]$ and separates $h_1$ from $h_{(4L+10)n}$
\end{proposition}

\begin{lemma}[Lemma 3.1 of \cite{petyt2023hyperbolic}]
\label{lack backtracking}
Let $h$ and $k$ be $L-$separated, and let $\alpha$ be a CAT(0) geodesic. If there exist $t_1<t_2<t_3<t_4$ satisfying either
\begin{align*}
 & \alpha(t_1)\in h, \alpha(t_2)\in k, \alpha(t_3)\in k, \alpha(t_4)\in h
 \\
 \text{or} \quad& \alpha(t_1)\in h, \alpha(t_2)\in k, \alpha(t_3)\in h, \alpha(t_4)\in k,
\end{align*}
then $t_3-t_2 \leq L+1$.
\end{lemma}

\begin{proposition}[Proposition 3.3 of \cite{petyt2023hyperbolic}]
\label{rough geodesic}
    If $\alpha$ is a CAT(0) geodesic and $x,y,z\in \alpha$ have $z\in[x,y]$, then $d_L(x,y)\ge d_L(x,z)+d_L(z,y)-3L-3$. In particular, for each $L$ there is a constant $q_L$, linear in $L$,, such that CAT(0) geodesics are unparametrized $q_L$-rough geodesics of $X_L$.
\end{proposition}

\begin{theorem}[Theorem 9.10 of \cite{petyt2023hyperbolic}]
Fix a sequence $(\lambda_L)$. There is a constant $\delta$ such that the curtain models of all CAT(0) spaces are $\delta$-hyperbolic.
\end{theorem}

\begin{theorem}[Theorem M of \cite{petyt2023hyperbolic}] 
\label{hyperbolicity}
A CAT(0) space $X$ is hyperbolic if and only if there exists $L_0$ such that for each $L\ge L_0$, the identity map is a quasi-isometry between each pair among $X,X_L$ and $X_D$.
\end{theorem}

We now prove the following two propositions for calculating the curtain model of particular types for our future use.

\begin{proposition}
    Let $Z=Y_1\times Y_2$ be a closed convex subset of a CAT(0) space $X$. Suppose $Y_i$ are both unbounded, then $Z\subseteq X_D$ is bounded.
\end{proposition}

\begin{proof}
Since $Z$ is convex and closed in $X$, $Z$ is CAT(0). Further, each $Y_i$ can be viewed as a convex closed subset of $Z$, so each $Y_i$ is again CAT(0). Let $(y_1,y_2),(\Tilde{y}_1,y_2)$ be two points with the same $Y_2$ coordinate in $Z$. Let $\alpha$ be the geodesic of $Y_1$ connecting $y_1$ and $\Tilde{y}_1$. Pick two disjoint curtains $h_1$ and $h_2$ in dual to the geodesic $\alpha\times\{y_2\}$, say at $(p,y_2)$ and $(q,y_2)$ respectively. We see that these two curtains cannot be $L$-separated. Indeed, we clearly have $[p-1/2,p+1/2]\times Y_2\subseteq h_1$ and $[q-1/2,q+1/2]\times Y_2\subseteq h_2$. Now pick any $s\in Y_1$ and a geodesic $\beta$ in $Y_2$ of length at least $10L$ (note $Y_2$ is unbounded). We see that any chain dual to the geodesic $\{s\} \times \beta$ must contain a subset of form $Y_1\times U$ for some $U\subseteq Y_2$ hence must intersect both $h_1$ and $h_2$. So $h_1$ and $h_2$ are not $L$-separated. Then we apply Proposition \ref{dualize} we see that $d_L((y_1,y_2),(\Tilde{y}_1,y_2))\leq 8L+20$. Now if $(y_1,y_2)$ and $(\Tilde{y}_1,\Tilde{y}_2)$ are two points in $Z$, then:
\[d_L((y_1,y_2),(\tilde{y}_1,\tilde{y}_2))\leq 16L+40,\]
by applying a symmetric argument and using triangle inequality.
Therefore $D((y_1,y_2),(\tilde{y}_1,\tilde{y}_2))\leq \sum_{L=1}^\infty \lambda_L(16L+40)<\infty$. Hence $Z$ is bounded in $X_D$.
\end{proof}

\begin{corollary}
\label{product}
Let $Z=X\times Y$, where $X$ and $Y$ are both unbounded CAT(0) space. Then the curtain model $Z_D$ is bounded.
\end{corollary}

\begin{proof}
    Note that the product of CAT(0) spaces is still CAT(0). So the corollary is simply a special case of the proposition above.
\end{proof}

\begin{proposition}
\label{paste}
Let $\{X_i,i\in I\}$ be a collection of CAT(0) spaces. Let $A$ be a compact convex metric space, equipped with a family of isometries $\phi_i: A \to A_i$ for $A_i \subset X_i$. Let $X=\bigsqcup_A X_i$ be the CAT(0) space obtained by gluing these $X_i$ along $A$. Then the curtain model for $X_D$ is quasi-isometric to $\bigsqcup_A(X_i)_D$.
\end{proposition}

\begin{proof}
We prove the case for $|I|=2$, and the proof for the general case is similar. Let $X=X_1\sqcup_A X_2$. There are 3 classes of curtains in $X$:
\begin{itemize}
    \item[A] curtains in $X_1$ dijoint from $A_1$.
    \item[B] curtains in $X_2$ disjoint from $A_2$.
    \item[C] curtains intersecting $A_1 \simeq A_2$
\end{itemize}

Let $N \ge \text{diam}(A) +3$ be a positive integer. The obvious but important observation is that in each chain, there can be at most $N$ curtains coming from class C. We fix $L \geq N+2$.

Let $x_1,y_1$ be two distinct points in $X_1$, and $c$ be a chain in $X$ realizing $d_L^X(x_1,y_1).$ Remove the curtains in $c$ from class C (if exist) to obtain $c'$. Then $c'$ is a $(L+N)$-chain in $X_1$. Since if not, then there exists a chain $\tilde{c}$ of cardinality $L+N+1$ intersecting all of $c'$. Note that $\tilde{c}$ can only contain curtains from classes A and C, and hence there must be at least $L+1$ of them coming from class A. Contradiction. So: $d_L^X(x_1,y_1)\leq d_{L+N}^{X_1}(x_1,y_1)+N$.

Conversely, let $c$ be a chain in $X_1$ realizing $d_L^{X_1}(x_1,y_1)$. Again, remove the curtains from class C (if exist) to obtain $c'$. Then $c'$ is a $(L+N)$-chain in $X$, by a similar argument in the first direction. So: $d_L^{X_1}(x_1,y_1)\leq d_{L+N}^X(x_1,y_1)+N$.

Combining both, we have the following estimation when $L \geq N+2$: 
\[
d_{L-N}^X(x_1,y_1)-N\leq d_L^{X_1}(x_1,y_1) \leq d_{L+N}^X(x_1,y_1)+N,
\]
where the superscript denotes which space we are computing the curtain distance in.

We can obtain similar result for points $x_2,y_2$ in $X_2$ when $L \geq N+2$
\[
d_{L-N}^X(x_2,y_2)-N\leq d_L^{X_2}(x_2,y_2) \leq d_{L+N}^X(x_2,y_2)+N.
\]

Finally, we deal with the distance of points from different original spaces. Let $x_1\in X_1$ and $x_2 \in X_2$. Choose $a \in A$, and set $a_i = \phi_i(a)$. Then we have:
\[
d_L^X(x_1,x_2)\leq d_L^X(x_1,a_1)+d_L^X(x_2,a_2) \leq d_{L+N}^{X_1}(x_1,p_1)+d_{L+N}^{X_2}(x_2,p_2)+2N.
\]

Conversely, let $c_i$ be a chain in $X_i$ realizing $d_{L-N}^{X_i}(x_i,a_i)$. Let $c_i'$ be obtained from $c_i$ by deleting all curtains intersecting $A_i$. We claim that $c'=c_1'\cup c_2'$ is a $L$-chain in $x$ separating $x_1$ and $x_2$. Indeed, let $h_1$ and $h_2$ be two curtains in $c'$. If they come from the same $c_i'$, then we're done by noting that there cannot be a chain of cardinality $L+1$ intersecting both, since at least $L-N+1$ of them must lie completely in the same original space. On the other hand, say $h_i \subseteq X_i$, then any curtain intersecting both must intersect $A$, which shows that any chain intersecting both has cardinality at most $N<L$. Therefore: $(d_{L-N}^{X_1}(x_1,a_1) -1-N)+(d_{L-N}^{X_2}(x_2,a_2)-1-N)\leq d_L^X(x_1,x_2)-1$, which rearranges to $d_{L-N}^{X_1}(x_1,a_1)+d_{L-N}^{X_2}(x_2,a_2) -2N-1\leq d_L^X(x_1,x_2)$.

Therefore indeed we have:
\[
d_{L-N}^{X_1}(x_1,a_1)+d_{L-N}^{X_2}(x_2,a_2)-2N-1 \leq d_L^X(x_1,x_2)
\leq d_{L+N}^{X_1}(x_1,a_1)+d_{L+N}^{X_2}(x_2,a_2)+2N.
\]

Combining all the estimations given above, we conclude that $X_D$ is quasi-isometric to $(X_1)_D \sqcup_A (X_2)_D$.
\end{proof}

\subsection{Translation Length and the Classification of Isometries}
This section reviews Gromov's classical classification of isometries on CAT(0) spaces and hyperbolic spaces, based on the notion of translation length.

\begin{definition}[Translation Length]
Let X be a metric space, and $g\in$ Isom($X$). We define the \emph{translation length} of $g$ as:
\[
\tau(g)=\lim_{n\rightarrow \infty}\frac{d(x,g^nx)}{n}.
\]
Then $\tau(g)$ is well-defined and is independent of the choice of base point $x$. Note that $\tau(g)\leq \inf_{x\in X} d(x,gx)$.
Thus we also define the \emph{minimal set}:
\[
\textup{Min}(g)=\{x\in X: d(x,gx)=\tau(g)\}.
\]
Call the isometry $g$ \emph{semisimple} if $\textup{Min}(g)\neq \emptyset$.
\end{definition}

\begin{remark}
There is a classical result (Exercise 6.6 of Chapter II in \cite{Bridson1999MetricSO}) justifying the terminology of \say{minimal set}: If $X$ is a metric space and $g$ is a semisimple isometry, then $\tau(g)=\inf_{x\in X} d(x,gx)$.
\end{remark}

\begin{definition}[Classification of Isometries of CAT(0) Spaces]
Let $X$ be a CAT(0) space, and $g\neq id$ be an isometry of $X$. Then precisely one of the following hold:
\begin{enumerate}
    \item Call $g$ \emph{elliptic} if it has a fixed point.
    \item Call $g$ \emph{parabolic} if $\text{Min}(g)=\emptyset$.
    \item Call $g$ \emph{hyperbolic} if $\text{Min}(g) \neq \emptyset$ and $\tau(g)>0$; equivalent, there is a geodesic line $l$ such that $g$ acts on $l$ by translation of $\tau(g)$.
\end{enumerate}
\end{definition}

\begin{remark}
In fact, if $g$ is an isometry of a CAT(0) space, we always have $\tau(g)=\inf_{x\in X} d(x,gx)$. We've seen the equality for $g$ being semisimple, and one can refer to \cite{ballmann2006manifolds} for the proof in the case of $g$ being parabolic. We will use this characterization of translation length for CAT(0) spaces in the future.
\end{remark}

\begin{definition}[Classification of Isometries of Hyperbolic Spaces]
    Let $X$ be a hyperbolic space, and $g\neq id$ be an isometry of $X$. Then precisely one of the following hold:
    \begin{enumerate}
        \item Call $g$ \emph{elliptic} if one (equivalently every) orbit of $\langle g\rangle$ is bounded.
        \item Call $g$ \emph{parabolic} if $\tau(g)=0$ and one (equivalent every) orbit of $\langle g\rangle$ is unbounded.
        \item Call $g$ \emph{loxodromic} if $\tau(g)>0$.
    \end{enumerate}
\end{definition}

\begin{remark}
Buyalo proved that for a hyperbolic Hadamard space, a parabolic isometry in the sense of CAT(0) space has zero translation length \cite{buyalo1998geodesics}. By Proposition 6.7 in chapter II.6 of \cite{Bridson1999MetricSO}, we also know isometries of a complete CAT(0) space with bounded orbits need to fix a point. These show that there is no confusion when talking about parabolic isometries when $X$ is a hyperbolic Hadamard space. In fact, the two classifications coincide in this particular case.
\end{remark}
We now wish to understand the isometries of the hyperbolic space $X_D$. The following theorem tells us Isom($X$)=Isom($X_D$) under mild conditions.

\begin{theorem}[Theorem I of \cite{petyt2023hyperbolic}]
Let $X$ be a proper CAT(0) space with the geodesic extension property. If any one of the following hold, then Isom($X$)=Isom($X_D$).
\begin{enumerate}
    \item $X$ admits a proper cocompact action by a group that is not virtually free.
    \item X is a tree that does not embed in $\R$.
    \item $X$ is one-ended.
\end{enumerate}
\end{theorem}

\subsection{Examples of the Curtain Model and their Isometries}

Even though the isometry groups coincide under certain circumstances, the properties of the same isometry on $X$ and $X_D$ may differ, for example, due to the collapsing of Euclidean flats in $X_D$. In this section, we aim to give a few examples of calculating the curtain model $X_D$ and identifying which class a particular isometry belongs to in both $X$ and $X_D$. But before doing that, we first note the following general facts:

\begin{itemize}
    \item An elliptic isometry in a CAT(0) space $X$ remains elliptic in $X_D$.
    \item Suppose $X$ is both hyperbolic and CAT(0), and $g$ is a parabolic isometry in $X$. Then $g$ remains parabolic in $X_D$, since by Theorem \ref{hyperbolicity} we know $X$ is quasi-isometric to $X_D$.
\end{itemize}

We now turn to the construction of particular spaces and their isometries:

\begin{example}
\label{hyperbolic_to_elliptic}
A hyperbolic isometry in a CAT(0) space $X$ can be elliptic in $X_D$. Let $X=\R^2$, $g(x,y)=(x+1,y)$. Then $g$ is clearly hyperbolic in $X$, but $D(x,y)\leq2$ for any $x,y\in X$. In particular, $g$ is elliptic in $X_D$.
\end{example}

\begin{example}
\label{parabolic_to_elliptic}
A strictly parabolic isometry in a CAT(0) space X can be elliptic in $X_D$. Let $X=\H^2\times \R$ with $\H^2$ with the upper half plane model. Set $g(z,t)=(z+1,t)$. Then $g$ is parabolic with translation length $0$. Now since $\H^2$ and $\R$ are both unbounded, by Corollary \ref{product} we see that $X_D$ is bounded, so $g$ is elliptic in $X_D$.
\end{example}

\begin{example}
A parabolic isometry in a non-hyperbolic CAT(0) space $X$ can be parabolic in $X_D$. Let $\H^2$ be the upper half plane model for the hyperbolic plane. Let $X$ be the space by attaching a copy of $\R^2$ at point $n+i$ on $\H^2$ by identifying $n+i\in \H^2$ and $(0,0) \in \R^2$ for each $n\in \Z$. Let $g$ be the isometry of $X$ induced by the isometry $z\mapsto z+1$ in $\H^2$ (translate each copy of $\R^2$ to the right by 1). Then $g$ is a parabolic isometry of translation length $0$ in $X$. Apply proposition \ref{paste} we see that $X_D$ is quasi-isometric to $\H^2$ and $g$ becomes a parabolic isometry in the sense of hyperbolic space.
\end{example}

\begin{example}
The unsatisfying aspect of the previous example is that the isometry group is tiny. In particular, the action in the space is far from being cocompact. We now consider the following example.
Let $X$ be the space built inductively from the following steps. We start from $\H^2$, and we attach a copy of $\R^2$ at each point of $\H^2$. We call this space $X_1$. Now suppose we've constructed $X_{n-1}$, to obtain $X_n$ we simply attach a copy of $\H^2$ at each point of each copy of $\R^2$ we attached in the last step, and then we attach $\R^2$ at each point on each copy of $\H^2$ we just attached. If we continue forever in this way, we obtain our desirable space $X$.

Clearly $X$ is a complete CAT(0) space, and Isom($\H^2$)$*\R^2$ acts naturally on $X$ by isometry. Note this action is cocompact (in fact transitive). Moreover, the isometry induced by $z\mapsto z+1$ on $\H$ is a strictly parabolic isometry on $X$, and it remains parabolic in $X_D$. However, unlike the previous example, our space $X$ is far from being proper.
\end{example}

\begin{example}
\label{helix}
A hyperbolic isometry in a non-hyperbolic CAT(0) space $X$ can be loxodromic in $X_D$. Let $X=\widetilde{\R^2\setminus \D^2}$, where $\D^2$ is the unit open disc in $\R^2$. Let $g$ be the isometry on $X$ induced by $1\in \Z=\pi_1(\R^2\setminus \D^2)$. Then $g$ is clearly hyperbolic with translation length $2\pi$. Let's now characterize the curtain model of $X_D$. 

It's not hard to see the lift of the boundary of $\D^2$ is a quasi-line: for any two curtains dual to the lift of the same interval on the boundary but on different "sheets", there cannot be a third curtain that crosses both. On the other hand, any other point in $X$ is of distance at most $3$ away from a point on the lift of the boundary of the disc. Indeed, take any point in $x\in X$ and project it down to a point $y$ on the lift of the boundary of  $\D^2$. Consider the line tangent to the lift of the boundary at $y$. This line bounds a half flat in $X$ and this gives an upper bound of the $L$-distance between $x$ and $y$. We see that the distance is of at most $3$ for each $L$ and this shows that $D(x,y)\leq 3$. Therefore $X_D$ is a quasi-line, and $g$ is loxodromic in $X_D$.
\end{example}

\begin{example}
\label{noncompletespace}
A parabolic isometry in an incomplete CAT(0) space can be loxodromic in $X_D$. We slightly modify Example \ref{helix} by considering the $X$ to be the universal cover of the Euclidean plane minus the closed unit disc. Then the same isometry $g$ as in \ref{helix} becomes a contracting parabolic isometry with positive translation length. And indeed, we see that $g$ acts loxodromically on $X_D$. It is worth noting that the parabolicity of $g$ comes from the incompleteness of our space $X$, and if we take the completion of $X$ we recover our original example \ref{helix}. 
\end{example}

\section{Semisimple Isometries}
In this section, we aim to give a full classification of the dynamics of the semisimple isometries of the curtain model. Clearly, if $g$ is elliptic in $X$, then it's also elliptic in $X_D$. It remains to see the dynamics of the hyperbolic(axial) isometries. We use $\tau_D(g),\tau_L(g)$ to denote the translation length of $g$ in the hyperbolic spaces $X_D,X_L$ respectively.

\begin{proposition}
[Proposition 9.5 of \cite{petyt2023hyperbolic}] 
\label{Droughgeodesic}
Let $X$ be a CAT(0) space and $z\in[x,y]$, and $\Lambda$ is the constant associated to the curtain model $X_D$. Then the following inequality holds:
\[
D(x,y)\geq D(x,z)+D(z,y)-6\Lambda.
\]
In particular, there is a constant $q$ such that CAT(0) geodesics are unparametrized $q$-rough geodesics.
\end{proposition}

\begin{proposition}
    Let $g$ be a hyperbolic isometry in CAT(0) space $X$. Then $g$ cannot be parabolic in $X_D$.
\end{proposition}

\begin{proof}
    Assume $g\in$ Isom($X$) is hyperbolic with axis $l$.
    If $l$ is bounded in $X_D$, then g clearly has a bounded orbit in $X_D$. Now assume $l$ is unbounded in $X_D$, then $\exists N$ such that $D(x,g^Nx)\geq 6\Lambda+1$. Therefore by Proposition \ref{Droughgeodesic} we see:
\begin{align*}
&D(x,g^{Nn}x)\geq D(x,g^Nx)+D(g^Nx,g^{2N}x+\cdots+D(g^{N(n-1)}x,g^{Nn}x)-6(n-1)\Lambda \\
&\geq nD(x,g^Nx)-6n\Lambda =
n(D(x,g^Nx)-6\Lambda) \geq n.
\end{align*}

Therefore:
\[
\tau_D(g)=\lim_{n\rightarrow \infty} \frac{1}{Nn}D(x,g^{Nn}x) \geq \frac{1}{N}.
\]
This shows $g$ is loxodromic in $X_D$.
\end{proof}

We now recall the notion of contracting that will help us identify loxodromic actions in $X_D$.

\begin{definition}[Contracting Set]
Let $X$ be a CAT(0) space. A subset $A\subset X$ is called \emph{$C$-contracting}, if for any open ball $B$ disjoint from $A$, a closest point projection $\pi_A(B)$ has diameter at most $C$. We call a set contracting if it's $C$-contracting for some $C$.
\end{definition}

\begin{definition}[Contracting Isometry]
\label{contracting isometry}
Let $X$ be a CAT(0) space, and $g\in$ Isom($X$). We call $g$ a \emph{contracting isometry} if there exists some $x\in X$ such that:

\begin{enumerate}
    \item $n\mapsto g^nx$ is a quasi-isometric embedding from $\Z$ to $X$.
    \item the orbit $\{g^nx:n\in \Z\}$ is $C$-contracting for some $C$.
\end{enumerate}
\end{definition}

\begin{remark}
It's easy to see that if $A$ and $B$ are of finite Hausdorff distance, then $A$ is contracting if and only if $B$ is.
From the above observation, we conclude that the property of contracting is independent of the choice of the base point. Therefore, a hyperbolic isometry $g$ is contracting if and only if it has a contracting axis.
\end{remark}

The final theorem will conclude the classification of dynamics of the semisimple isometries of the curtain model.

\begin{theorem}
\label{hyperbolic contracting}
Let $g$ be a semisimple isometry of CAT(0) space $X$. The following are equivalent:
\begin{enumerate}
    \item $g$ is contracting.
    \item $g$ acts loxodromically on $X_L$, for some $L$.
    \item $g$ acts loxodromically on $X_D$.
\end{enumerate}
\end{theorem}

\begin{proof}
The equivalence of 1 and 2 is shown in Theorem 4.9 of \cite{petyt2023hyperbolic}. The direction $2 \implies 3$ is by definition. It remains to see if $g$ acts loxodromically on $X_D$, then $g$ acts loxodromically on $X_L$ for some $L$.

Assume not, then $\tau_L(g)=0$ for all $L$. Fix $N>1$, we have:

\[
\frac{1}{n}D(x,g^nx)=\frac{1}{n}\sum_{L=1}^\infty \lambda_L d_L(x,g^nx) \leq \sum_{L=1}^N \lambda_L \frac{d_L(x,g^nx)}{n}+\frac{d_\infty(x,g^nx)}{n}\sum_{L=N+1}^\infty \lambda_L \rightarrow \tau(g)\cdot \sum_{L=N+1}^\infty \lambda_L.
\]
as $n\rightarrow \infty$. But the choice of the $N$ is arbitrary we see indeed $\frac{1}{n}D(x,g^nx) \rightarrow 0$ as $n \rightarrow \infty$, and so $g$ is not loxodromic on $X_D$. Contradiction.
\end{proof}

\begin{remark}
    In fact, the equivalence of 1 and 3 can also deduced by Corollary 9.13 of \cite{petyt2023hyperbolic}.
\end{remark}

\begin{example}
The above theorem provides us with another way to characterize $X_D$ and classify the type of isometry in Example \ref{helix}. For any ball $B$ disjoint from the lift of the boundary $l$, the closest point projection of the ball towards $l$ has diameter at most $2\pi$, that is to say, $l$ is a contracting axis of $g$. So by the above theorem, we see $g$ acts loxodromically on $X_D$.
\end{example}
We summarize the result of this section by the following theorem:

\begin{theorem}
Let $g$ be a semisimple action of CAT(0) space $X$. Then $g$ cannot be parabolic in $X_D$. More precisely, $g$ acts loxodromically on $X_D$ if and only if $g$ is contracting, otherwise $g$ is elliptic in $X_D$.
\end{theorem}

\section{Parabolic Isometries}

The main goal of this section is to answer the following question: If $g$ is a parabolic isometry of $X$ with $\tau(g)>0$, what is the type of $g$ in Isom($X_D$)?

\subsection{Loxodromic Actions on the Curtain Model}

 We will prove that $g$ is loxodromic in the curtain model if and only if $g$ is contracting in the original space. We say two subsets of a CAT(0) space $X$ are \emph{$C$-Hausdorff equivalent} if each is contained in the $C$-Hausdorff neighborhood of each other. We recall the notion of rank-one isometries introduced by Bestvina and Fujiwara:

\begin{definition}[Definition 5.1 of \cite{bestvina2008characterization}]
    Let $X$ be a CAT(0) space and $g\in$ Isom($X$) be an isometry and let $x_0\in X$ be a base point. Denote $x_n=g^nx_0$. We say $g$ is \emph{rank-one} if $d(x_0,x_n)\rightarrow\infty$ as $n\rightarrow \infty$ and there exists $C>0$ such that for every $n>0$
    \begin{enumerate}
        \item the geodesic $[x_0,x_n]$ is $C$-Hausdorff equivalent to $\{x_0,\cdots,x_n\}$, and
        \item $[x_0,x_n]$ is $C$-contracting.
    \end{enumerate}
\end{definition}
 
It turns out the above notion of rank-one coincides with our notion of contracting in Definition \ref{contracting isometry}. Assume that $g$ is contracting with constant $C$. Fix $x\in X$. Then the orbit $\{g^nx: n\in \Z\}$ is
quasi-convex in $X$ and therefore $n\mapsto x_n=g^nx$ is a quasi-isometric embedding. Hence for each $n>0$, $\bigcup_{i=0}^{n-1}[x_i,x_i+1]$ is a quasi-geodesic with Hausdorff distance not exceeding $D$ from $[x_0,x_n]$, for some constant $D$. This shows that $[x_0,x_n]$ is again uniformly contracting with constant depending on $C$ and $D$. The reverse direction is clear again because the property of being contracting is invariant under a perturbation of a finite Hausdorff distance.

The basic technique we'll use is the following result:

\begin{theorem}[Theorem 4.2 of \cite{petyt2023hyperbolic}]
\label{lox}
Let $X$ be a CAT(0) space. If $\alpha \subset X$ is a $D$-contracting geodesic, then there is a $(10D+3)$-chain of curtains met $8D$-frequently by $\alpha$. Conversely, if a geodesic $\beta$ meets an $L$-chain of curtains $T$-frequently for $T\ge1$, then $\beta$ is $16T(L+4)$-contracting.
\end{theorem}

\begin{proposition}
Let X be a CAT(0) space, and $g\in$ Isom($X$) is parabolic such that $\tau(g)>0$. Fix any $x\in X$, and let $x_n=g^nx$. If $[x_0,x_n]$ is uniformly contracting, then $g$ is loxodromic in $X_L$ for some $L$.
\end{proposition}

\begin{proof}
Suppose $[x_0,x_n]$ is $K$-contracting for some constant K. From theorem \ref{lox}, there exists a $L=(10K+3)$ chain meets $8K$-frequently by $[x_0,x_n]$, therefore $d_L(x_0,x_n) \geq \frac{1}{8K}d(x_0,x_n)-10 \geq \frac{n}{8K}\tau(g)-10$. Therefore $lim_{n\rightarrow \infty}\frac{1}{n}d_L(x_0,x_n)\geq \frac{1}{8K}\tau(g)>0$. This shows that $g$ is loxodromic in $X_L$.
\end{proof}

To prove the converse, we adapt the proof of Proposition 4.12 of \cite{petyt2023hyperbolic} on stable subgroups in the following particular setting:

\begin{proposition}
Let X be a CAT(0) space, and $g\in$ Isom($X$) is parabolic such that $\tau(g)>0$. If $g$ is loxodromic in $X_L$ for some $L$, then $g$ is contracting.
\end{proposition}

\begin{proof}
Consider the following map:
$\Z \xrightarrow{\phi} X \xrightarrow{\iota} X_L$, where $\phi:n\mapsto g^nx$ and $\iota$ is just the remetrization. Since $g$ is loxodromic in $X_L$, $\iota \circ \phi$ is a quasi-isometric embedding.
Because $\phi$ is Lipschitz map and $\iota$ is coarsely Lipschitz, this means that $\{g^nx:n\in \Z \}$ is also quasi-isometrically embedded in $X$. In particular, for any $n,m\in \Z$ the CAT(0) geodesic $[g^nx,g^mx]$ is uniformly quasi-isometrically embedded in $X_L$, and so is contracting by \ref{lox}. Since being contracting is equivalent to being Morse in the sense of CAT(0) space \cite{sultan2011hyperbolic}, it implies that the quasi-isometric image of any $\Z$-geodesic from $g^nx$ to $g^mx$ is uniformly Hausdorff-close to $[g^nx,g^mx]$. Hence $[g^nx,g^mx]$ stays uniformly Hausdorff-close to $\langle g \rangle x$. This shows that $g$ is rank-one, or equivalently contracting.
\end{proof}

We summarize our previous results in the following theorem:

\begin{theorem}
\label{duality}
Let $X$ be a CAT(0) space, and $g\in$ Isom($X$) be parabolic with $\tau(g)>0$. Then the following are equivalent:
\begin{enumerate}
    \item $g$ is contracting.
    \item there exists some $x\in X$ such that $[x,g^nx]$ is uniformly contracting.
    \item $g$ acts loxodromically on $X_L$ for some $L$.
    \item $g$ acts loxodromically on $X_D$.
\end{enumerate}
\begin{proof}
$(1)\implies(2)$ is by definition. We've shown $(2)\implies(3)\implies(1)$. The equivalence of $(3)$ and $(4)$ follows from the same proof as in \ref{hyperbolic contracting}.
\end{proof}
\end{theorem}

\begin{remark}
The interesting part of the above theorem is the equivalence of $g$ being contracting and $[x,g^nx]$ being uniformly contracting when $g$ is a parabolic isometry with $\tau(g)>0$. This equivalence doesn't involve our new model $X_D$ and is a general fact about the CAT(0) space. 

It is worth noting that the former of course implies the latter by definition, but the converse is not always true if we drop the condition $\tau(g)>0$. For example, we consider $X=\H^2$ and $g:z\mapsto z+1$. Then $g$ clearly doesn't have a contracting orbit but for any $x,n$ we have $[x,g^nx]$ is uniformly contracting, since any geodesic in the hyperbolic plane is uniformly contracting.
\end{remark}

The following corollary is immediate after we include the trivial case of $g$ being hyperbolic:

\begin{corollary}
    Let $X$ be a CAT(0) space, and $g\in$ Isom($X$) with $\tau(g)>0$. Then $g$ is contracting if and only if there exists $x\in X$ such that $[x,g^nx]$ is uniformly contracting.
\end{corollary}

\subsection{Sublinear Axis}
When we look back at our classification of the dynamics of semisimple actions, we made essential use of the "natural objects" associated with the actions, say the axes of hyperbolic actions. In this section, we will introduce a \say{natural object} to look at if $g$ is a parabolic isometry with positive translation length. Let's first recall the following definitions.

\begin{definition}[Visual Boundary, Definition 8.1 of Chapter II in \cite{Bridson1999MetricSO}]
Let $X$ be a metric space. Two geodesics rays $\alpha,\alpha':[0,\infty)\rightarrow X$ are said to be \emph{asymptotic} if there exists a constant $K$ such that $d(
(\alpha(t),\alpha(t'))\leq K$ for all $t\ge 0$. The set $\partial X$ of \emph{boundary points} of $X$ is the set of equivalence class of geodesic rays - two rays are equivalent if and only if they're asymptotic. If $\alpha$ is a geodesic ray, then we denote $\alpha(\infty)$ as its corresponding equivalence class in $\partial X$.
\end{definition}

\begin{definition}[Visibility Space, Definition 9.28 of Chapter II in \cite{Bridson1999MetricSO}]

Let $X$ be a CAT(0) space, and $\partial X$ be its visual boundary. A point $\xi \in \partial X$ is said to be a \emph{visibility point} if for any other $\eta \in \partial X$, there is a geodesic $\alpha$ such that 
$\alpha(-\infty)=\eta$ and $\alpha(\infty)=\xi$. We say $X$ is a \emph{visibility space} if each boundary point is a visibility point.

\end{definition}
We consider a special case of Theorem 2.1 of \cite{KarlssonMargulis} and follow the notations there by setting: the measure space to be a single point $p$, 
$L=$ identity map, $\omega(x)=g$ and  $D=X$. Here $X$ is a complete CAT(0) space and $g$ is a parabolic isometry of $X$. We then obtain the following theorem:
\begin{theorem}[Theorem 2.7 of \cite{Wu2017}, Karlsson-Margulis] 
\label{sublinear-axis}
Let $X$ be a complete CAT(0) space with a base point $p\in M$ and $g$ be a parabolic isometry with $\tau(g)>0$. Then there exist a unique $\xi \in \partial X$ which is fixed by $g$, and a geodesic ray $\alpha$ such that $\alpha(0)=p$, $\alpha(\infty)=\xi$, and
\[
\lim_{n\rightarrow \infty}\frac{d(g^np,\alpha(\tau(g)\cdot n))}{n}=0.
\]
\end{theorem}

We call $\alpha$ the sublinear axis of the isometry $g$ with endpoint $\xi$ at infinity. With the help of this natural object, Wu proved the following:

\begin{theorem}[Theorem 1.2 of \cite{Wu2017}]
\label{structure}
Let $X$ be a complete, visibility CAT(0) space. Then a parabolic isometry $g$ has $\tau(g)>0$ if and only if there exists a $g$-invariant subset $U\times \R$ which is closed and convex in $X$, such that $g$ acts on $U\times \R$ as 

\[
g \circ (x,t)=(g_1(x),t+t_0) \qquad \forall (x,t)\in U\times \R
\]
where $g_1$ is a parabolic isometry with translation length $0$ on $U$.
\end{theorem}

From Theorem \ref{structure}, we can determine the behavior of a parabolic isometry $g$ with $\tau(g)>0$ in the curtain model of a complete visibility CAT(0) space:

\begin{corollary}
\label{complete+visibility}
Let $X$ be a complete visibility CAT(0) space, and $g\in$ Isom($X$) is parabolic with $\tau(g)>0$. Then $g$ is an elliptic element in the curtain model $X_D$.
\end{corollary}

\begin{proof}
From the preceding theorem, we see that we can find an infinite flat strip $U\times \R$ which is a $g$-invariant convex closed subset of $X$ on which $g$ splits as in the theorem. Note that $g_1$ is a parabolic isometry with translation length $0$ on $U$. In particular, this means that $U$ must be unbounded. Therefore, by Corollary \ref{product}, the curtain model for $U\times \R$ is bounded, say has diameter $R$ since both $U$ and $\R$ are unbounded. Pick $x\in U\times \R$, then $D(x,g^nx) \leq R$ since $U \times \R$ is $g-$invariant. Hence $g$ must be elliptic in $X_D$.
\end{proof}

In fact, in the next two subsections, we will prove a similar result under a more general condition that $X$ is a proper and complete CAT(0) space, by splitting into two cases on whether the corresponding sublinear axis is bounded or not.

\subsection{Bounded Sublinear Axis}

We first deal with the case that the sublinear axis is bounded in $X_L$ under $g$. We say a curtain $h$ touches a geodesic $\alpha$ if $h\cap \alpha \neq \emptyset$; we say $h$ crosses $\alpha$ if there exists $p,q\in \alpha$ such that $p\in h^-$ and $q\in h^+$.

\begin{proposition}
Let $X$ be a complete CAT(0) space with base point $p$, and $g$ be a parabolic isometry with $\tau(g)>0$. Let $\alpha$ and $\xi=\alpha(\infty)$ be determined by Theorem \ref{sublinear-axis}. Fix $L\ge 1$. If $\alpha$ is bounded in $X_L$, then $g$ is elliptic in $X_L$.
\end{proposition}

\begin{proof}

We'll show that if $\alpha$ is bounded in $X_L$, then $p$ has a bounded orbit in $X_L$.

Assume that $\alpha$ has diameter $K+1$ for some $K$, but $p$ has unbounded orbit in $X_L$. Then we can find $n$ such that $d_L(p,g^np)\ge 2K+2L+9$. Suppose $h'_1,\cdots,h'_{2K+2L+8}$ realize this. Note that there can be at most $K$ of such curtains crossing $\alpha$ once, and there can be at most one curtain that touches but not crosses $\alpha$. Moreover, by Lemma \ref{lack backtracking} we see there can be at most $L+2$ curtains that cross $\alpha$ at least twice. Therefore in summary there can be at most $K+L+3$ curtains that touch $\alpha$. Similarly, there can be at most $K+L+3$ curtains touching $g^n\alpha$. 

This means there will be at least two curtains, say $h'_1$ and $h'_2$, that intersect neither $\alpha$ nor $g^n\alpha$. Further, because $h'_1$ and $h'_2$ separates $p$ and $g^np$, we see that $h'_1$ and $h'_2$ actually both separate $\alpha$ and $g^n\alpha$. Note also since $\alpha$ and $g^n\alpha$ are asymptotic by Theorem \ref{sublinear-axis}, we know $d(\alpha(t),g^n\alpha(t))\leq M$ for some $M$ and all $t\ge0$.

We now inductively construct a chain of disjoint curtains $(h_i)$ dual to $\alpha$ that meets both $h'_1$ and $h'_2$. For the base case, pick $y_1\in g^n\alpha(\R)$ such that $d(y_1,p)\ge M+2$, and suppose $x_1$ be its closest point projection onto $\alpha$. Then $d(x_1,p)\ge d(p,y_1)-d(x_1,y_1)\ge(M+2)-M=2$, therefore we can take $h_1$ to be the curtain dual to $\alpha$ at $x_1$. Indeed $h_1$ contains $[x_1,y_1]$ and $[x_1,y_1]$ intersects $h'_1$, so $h_1$ meets $h'_1$. For a similar reason, we see $h_1$ meets $h'_2$.

Suppose we've constructed $h_1,\cdots,h_{m-1}$, and suppose $h_{m-1}$ is dual to $\alpha$ at $x_{m-1}$. We pick $y_m\in g^n\alpha(\R)$ after all the intersections of the previous curtains with $g^n\alpha$ and ensure that $d(y_m,[p,x_{m-1}])\ge M+2$.Still, let $x_m$ be the closest point projection of $y_m$ onto $\alpha$. We see as before that $d(x_m,x_{m-1})\ge d(y_m,x_{m-1})-d(y_m,x_m) \ge (M+2)-M=2$, and $x_m$ must lie after $x_{m-1}$ on $\alpha$. We let $h_m$ to be the curtain dual to $\alpha$ at $x_m$, then $h_m$ is indeed disjoint from $h_1,\cdots,h_{m-1}$ and contains $[x_m,y_m]$, hence meets $h'_1$ and $h'_2$.

Therefore, we can find an infinite chain that meets both $h'_1$ and $h'_2$, contradicting the fact that $h'_1$ and $h'_2$ are $L$-separated. Hence our assumption that $p$ has an unbounded orbit is wrong. So $g$ has a bounded orbit in each $X_L$.
\end{proof}

Note that in the proof of the above proposition we give a linear control of the $L$-distance on the orbit of $p$ by the diameter of $\alpha$ in $X_L$. We'll see that this enables us to pass the boundedness of the orbit in $X_L$ to the boundedness in $X_D$.

\begin{proposition}
\label{bounded sublinear axis}
Let $X$ be a complete CAT(0) space with base point $p$, and $g$ be a parabolic isometry with $\tau(g)>0$. Let $\alpha$ and $\xi=\alpha(\infty)$ be determined by Theorem \ref{sublinear-axis}. If $\alpha$ is bounded in $X_D$, then $g$ is elliptic in $X_D$.
\end{proposition}

\begin{proof}
$\alpha$ is bounded in $X_D$, say by $K$. We know $\alpha$ is bounded in $X_L$. Assume that $\alpha$ has diameter $K_L$ in $X_L$. Then according to the proof of the above proposition we see $d_L(p,g^np)\leq 2K_L+2L+8$ for each $L$. Suppose $a_L,b_L\in\alpha$ realizes $K_L$: $d_L(a_L,b_L)=K_L$. Fix $N>0$. Take $b=\max_{1 \leq L \leq N}b_L$. Then:
\begin{align*}
&K \ge D(p,b) \ge \sum_{L=1}^N \lambda_Ld_L(p,b) \ge \sum_{L=1}^N \lambda_L(d_L(p,a_L)+d_L(a_L,b_L)+d(b_L,b)-(6L+6))\\
&\ge \sum_{L=1}^N \lambda_L(K_L-6L-6),
\end{align*}
by Proposition \ref{rough geodesic}. 

But the choice of $N$ is arbitrary, therefore $K\ge \sum_{L=1}^\infty \lambda_L(K_L-6L-6) \implies \sum_{L=1}^\infty \lambda_LK_L \leq K+\sum_{L=1}^\infty \lambda_L(6L+6) < \infty$. This shows that for any $n>0$, we have: 
\begin{align*}
&D(p,g^np)=\sum_{L=1}^\infty \lambda_L d_L(p,g^np) \leq \sum_{L=1}^\infty \lambda_L(2K_L+2L+8) = 8+2\sum_{L=1}^\infty \lambda_LK_L +2\sum_{L=1}^\infty L\lambda_L\\
& \leq 8+2(K+\sum_{L=1}^\infty \lambda_L(6L+6))+2\sum_{L=1}^\infty L\lambda_L< \infty.
\end{align*}
So $g$ is elliptic in $X_D$.
\end{proof}

\subsection{Unbounded Sublinear Axis}

We use $\partial X_D$ to denote the Gromov boundary of the hyperbolic space $X_D$. In view of the fact that CAT(0) geodesics are unparametrized rough geodesics of $X_D$, we can consider the subspace $\partial_D X$ of the visual boundary $\partial X$ consisting of all points represented by rays that also represent points in $\partial X_D$. The following theorem reveals the strong connection between these two spaces when $X$ is proper:

\begin{theorem}[Theorem 9.19 of \cite{petyt2023hyperbolic}]
\label{boundary}
If $X$ is proper CAT(0) space, then $\partial_DX$ is Isom($X$)-invariant, and its points are all visibility points in $\partial X$. Moreover, the identity map induces an Isom($X$)-equivariant homeomorphism $\partial_DX \rightarrow \partial X_D$.
\end{theorem}

The key observation is the following: in the first half of the proof of Theorem \ref{structure}, the strong condition that $X$ is visibility wasn't fully used. Instead, to construct the $g$-invariant infinite flat strip $U\times \R$, we only need the visibility of a particular point, namely $\xi$ in Theorem \ref{sublinear-axis}, which can be implied by Theorem \ref{boundary} under suitable conditions. We first mark down the following proposition that plays an important role in the proof of Theorem \ref{structure}.

\begin{proposition}[Proposition 3.1 of \cite{Wu2017}]
\label{fix2points}
Let $X$ be a complete CAT(0) space and $g$ be an isometry with $\tau(g)>0$, then $g$ fixes at least two points on $\partial X$.
\end{proposition}

We can now adapt the proof of Theorem \ref{structure} to prove the following: 

\begin{proposition}
\label{unbounded sublinear axis}
Let $X$ be a complete and proper CAT(0) space with base point $p$, and $g$ be a parabolic isometry with $\tau(g)>0$. Let $\alpha$ and $\xi=\alpha(\infty)$ be determined by Theorem  \ref{sublinear-axis}. If $\xi \in \partial_DX$, then $g$ is elliptic in $X_D$.
\end{proposition}

\begin{proof}
By Theorem \ref{boundary} we know that $\xi$ is a visibility point of $\partial X$.
From Proposition \ref{fix2points}, we can find another point $\eta \neq \xi$ on $\partial X$ such that $\eta$ is a fixed point of $g$. Because $\xi$ is a visibility point, there is a geodesic $\alpha:\R \rightarrow X$ such that $\alpha(-\infty)=\xi$ and $\alpha(+\infty)=\eta$. Let $P_\alpha$ be the set of geodesic lines asymptotic to $\alpha$. By Theorem 2.14 of chapter II in \cite{Bridson1999MetricSO}, we know $P_\alpha$ is isometric to the product $U \times \R$ where $U$ is a closed convex subset in $X$.

Since $\alpha(-\infty),\alpha(+\infty)$ are both fixed by $g$, $g(\alpha(\R))$ is also geodesic line which is asymptotic to $\alpha(\R)$. In particular $P_\alpha=U\times \R$ is a $g$-invariant subset in $X$. By Proposition 5.3 of chapter I in \cite{Bridson1999MetricSO}, we know $g$ splits as $g=(g_1,g_2)$ where $g_1$ is an isometry of $U$ and $g_2$ is an isometry of $\R$.

Now, if $U$ is bounded, then Theorem 6.7 of chapter II in \cite{Bridson1999MetricSO} tells us that there exists $x_1\in X$ such that $g_1 x_1=x_1$. Since $g_2$ acts on $\R$ by isometry, it must be semisimple. Therefore $g=(g_1,g_2)$ is semisimple. Contradiction. 

Therefore $U$ is unbounded. Now $g$ acts on the product space $U\times \R$ with both $U$ and $\R$ unbounded. By applying the same argument as in Corollary \ref{complete+visibility} we see that indeed $g$ is elliptic in $X_D$.
\end{proof}

\begin{remark}
A direct consequence of the above proposition is the following implications: $\alpha$ is unbounded in $X_L$ for some $L$ $\implies$ $\alpha$ is unbounded in $X_D$  $\implies$ $g$ is elliptic in $X_D$ $\implies$ $g$ is elliptic in $X_L$ for every $L$.
\end{remark}

Combining Proposition \ref{bounded sublinear axis} and Proposition \ref{unbounded sublinear axis}, we can deduce the following theorem:

\begin{theorem}
Let $X$ be a complete, proper CAT(0) space and $g$ be a parabolic isometry with $\tau(g)>0$. Then $g$ is elliptic in $X_D$.
\end{theorem}

\begin{corollary}
    \label{exclude_contracting}
    Let $X$ be a complete, proper CAT(0) space and $g$ be a parabolic isometry with $\tau(g)>0$. Then $g$ cannot be contracting.
\end{corollary}

\begin{proof}
We know $g$ is elliptic in $X_D$, and therefore $g$ is not contracting by Theorem \ref{duality}. 
\end{proof}

\begin{example} This is an example from \cite{Wu2017}.
Let $X$ be the upper half plane with Riemannian metric $ds^2=(dx^2+dy^2)+\frac{dx^2+dy^2}{y^2}$. Note this metric is the sum of Euclidean metric and hyperbolic metric, both of which are complete. Therefore $(X,ds^2)$ is indeed a complete metric space. $X$ is also clearly proper since the topology on $X$ is precisely the Euclidean topology. One can calculate the sectional curvature $K(x,y)=-\frac{1+3y^2}{(1+y^2)^3}<0$. Therefore $X$ is indeed a complete, proper CAT(0) space.

Now let $g:z\mapsto z+1$ be an isometry of $X$. One can easily see that $\tau(g)=1$ and it cannot be obtained. Hence by Corollary \ref{exclude_contracting}, $g$ is not contracting.
\end{example}

\subsection*{Declaration}
\textbf{Conflict of Interest} \quad All authors certify that they have no affiliations with or involvement in any organization or entity with any financial interest or non-financial interest in the subject matter or materials discussed in this manuscript.
\\
\textbf{Funding} \quad This research was partially supported by University of Oxford Summer Research Bursary for summer 2023.
\\
\textbf{Data Availability Statement} \quad Data sharing does not apply to this article as no datasets were generated or analyzed during the preparation of this article.

\printbibliography

@misc{petyt2023hyperbolic,
      title={Hyperbolic models for CAT(0) spaces}, 
      author={Harry Petyt and Davide Spriano and Abdul Zalloum},
      year={2023},
      eprint={2207.14127},
      archivePrefix={arXiv},
      primaryClass={math.MG}
}

@misc{bestvina2008characterization,
      title={A characterization of higher rank symmetric spaces via bounded cohomology}, 
      author={Mladen Bestvina and Koji Fujiwara},
      year={2008},
      eprint={math/0702274},
      archivePrefix={arXiv},
      primaryClass={math.GR}
}

@article{Wu2017,
author = {Wu, Yunhui},
year = {2017},
month = {02},
pages = {},
title = {Translation Lengths of Parabolic Isometries of CAT(0) Spaces and Their Applications},
volume = {28},
journal = {The Journal of Geometric Analysis},
doi = {10.1007/s12220-017-9824-1}
}

@inproceedings{Bridson1999MetricSO,
  title={Metric Spaces of Non-Positive Curvature},
  author={Martin R. Bridson and Andr'e Haefliger},
  year={1999},
  doi={https://doi.org/10.1007/978-3-662-12494-9}
}

@misc{norin2022torsion,
      title={Torsion groups do not act on $2$-dimensional $\mathrm{CAT}(0)$ complexes}, 
      author={Sergey Norin and Damian Osajda and Piotr Przytycki},
      year={2022},
      eprint={1902.02457},
      archivePrefix={arXiv},
      primaryClass={math.GR}
}

@misc{sultan2011hyperbolic,
      title={Hyperbolic quasi-geodesics in CAT(0) spaces}, 
      author={Harold Mark Sultan},
      year={2011},
      eprint={1112.4246},
      archivePrefix={arXiv},
      primaryClass={math.GT}
}

@article{KarlssonMargulis,
author = {Karlsson, Anders and Margulis, Gregory},
year = {1999},
month = {12},
pages = {107-123},
title = {A Multiplicative Ergodic Theorem and Nonpositively Curved Spaces},
volume = {208},
journal = {Communications in Mathematical Physics},
doi = {10.1007/s002200050750}
}

@article{Tits1972FreeSI,
  title={Free subgroups in linear groups},
  author={Jacques Tits},
  journal={Journal of Algebra},
  year={1972},
  volume={20},
  pages={250-270},
  url={https://api.semanticscholar.org/CorpusID:122283776}
}

@incollection{gromov1987hyperbolic,
  title={Hyperbolic groups},
  author={Gromov, Mikhael},
  booktitle={Essays in group theory},
  pages={75--263},
  year={1987},
  publisher={Springer}
}

@article{mccarthy1985tits,
  title={A “Tits-alternative” for subgroups of surface mapping class groups},
  author={McCarthy, John},
  journal={Transactions of the American Mathematical Society},
  volume={291},
  number={2},
  pages={583--612},
  year={1985}
}

@article{buyalo1998geodesics,
  title={Geodesics in Hadamard spaces},
  author={Buyalo, Sergey Vladimirovich},
  journal={Algebra i Analiz},
  volume={10},
  number={2},
  pages={93--123},
  year={1998},
  publisher={St. Petersburg Department of Steklov Institute of Mathematics, Russian~…}
}

@inproceedings{ballmann2006manifolds,
  title={Manifolds of non positive curvature},
  author={Ballmann, Werner},
  booktitle={Arbeitstagung Bonn 1984: Proceedings of the meeting held by the Max-Planck-Institut f{\"u}r Mathematik, Bonn June 15--22, 1984},
  pages={261--268},
  year={2006},
  organization={Springer}
}

@misc{sageev2004titsalternativecat0cubical,
      title={The Tits alternative for CAT(0) cubical complexes}, 
      author={Michah Sageev and Daniel T. Wise},
      year={2004},
      eprint={math/0405022},
      archivePrefix={arXiv},
      primaryClass={math.GR},
      url={https://arxiv.org/abs/math/0405022}, 
}

\end{document}